\documentclass[reqno]{amsart}
\usepackage[active]{srcltx}
\usepackage{latexsym}
\usepackage{amssymb,amsmath,amsfonts}
\usepackage[ansinew]{inputenc}
\usepackage{graphicx}
\usepackage{color}
\usepackage{url}
\usepackage[colorlinks]{hyperref}
\usepackage{pdftricks}
\numberwithin{equation}{section}
\newtheorem{thm}{Theorem}[section]
\newtheorem{cor}[thm]{Corollary}
\newtheorem{lem}[thm]{Lemma}

\theoremstyle{definition}

\theoremstyle{remark}

\newtheorem*{themz}{{\bf Theorem Z}}
\newtheorem*{themcs}{{\bf Theorem CS}}

\begin{document}
\vspace*{1.5cm}
\title{Strong convergence for the modified Mann's iteration of $\lambda$-strict pseudocontraction}%
\thanks{The Corresponding author email: songyisheng1@gmail.com(Song)}%
\thanks{This work is supported partly by the National Natural Science Foundation of P.R. China(11071279,11171094,11271112) and by the Research Projects of Department of Science and Technology of Henan Province (Grant No. 122300410414,132300410432)}
\maketitle
\begin{center}Yisheng Song$^{1,2}$ and Hongjun Wang$^1$\\
1. College of Mathematics and Information Science,
Henan Normal University, XinXiang, P.R. China, 453007\\
2. Department of Applied Mathematics, The Hong Kong Polytechnic University, Hung Hom, Kowloon, Hong Kong
\end{center}

\begin{quote}{\bf Abstract.}\  In this paper, for an $\lambda$-strict pseudocontraction $T$, we
 prove strong convergence of the modified Mann's iteration defined by  $$x_{n+1}=\beta_{n}u+\gamma_nx_n+(1-\beta_{n}-\gamma_n)[\alpha_{n}Tx_n+(1-\alpha_{n})x_n],$$ where $\{\alpha_{n}\}$, $ \{\beta_{n}\}$ and $\{\gamma_n\}$
 in $(0,1)$ satisfy:

(i) $0 \leq \alpha_{n}\leq \frac{\lambda}{K^2}$ with $\liminf\limits_{n\to\infty}\alpha_n(\lambda-K^2\alpha_n)> 0$;

(ii) $\lim\limits_{n\to\infty}\beta_n= 0$ and
$\sum\limits_{n=1}^\infty\beta_n=\infty$;

(iii) $\limsup\limits_{n\to\infty}\gamma_n<1$. \\
Our results unify and improve some existing results.\\
 {\bf Key Words and Phrases:} $\lambda$-strict pseudocontraction, modified Mann's iteration, 2-uniformly smooth Banach space.\\
{\bf 2000 AMS Subject Classification:}  49J40, 47H05, 47H04, 65J15,
47H10.
\end{quote}
\vspace*{1.0cm}
\pagestyle{myheadings}
\thispagestyle{plain}
\markboth{Yisheng Song and Hngjun Wang}{Modified Mann's iteration of $\lambda$-strict pseudocontraction}
\section{\bf Introduction}
Throughout this paper, let $E$ be a Banach space with the norm $\|\cdot\|$ and  the dual space $E^*$ and $\langle y,
x^*\rangle$ denote the
value of $x^*\in E^*$ at $y\in E$.  The normalized duality mapping $J$ from $E$ into
$2^{E^{*}}$ is defined by the following equation:
 $$J(x)=\{x^*\in E^*:\langle x,x^*\rangle=\|x\|\|x^*\|,\|x\|=\|x^*\|\}.$$  Let $F(T)=\{x\in E:Tx=x\}$, the set of all fixed  point  of a mapping $T$.

  Recall that a mapping $T$ with domain $D(T)$ and range $R(T)$ in Banach space $E$
  is called {\em Lipschitzian}
   if  there exists $L>0$ such that
   $$\| Tx-Ty\| \leq L\|x-y\|\mbox{ for all }x,y\in D(T).$$  $T$ is said to be {\em nonexpansive} if
   $L=1$ in the above inequality. $T$ is called {\em $\lambda$-strictly pseudocontractive}
 if there exists $\lambda\in (0,1) $ and $j(x-y)\in J(x-y)$ such that
     \begin{equation}\label{eq:01}
      \langle Tx-Ty,j(x-y)\rangle \leq ||x-y||^{2}- \lambda||x-y-(Tx-Ty)||^{2} \mbox{ for all }x,y\in D(T).
      \end{equation} $T$ is called {\em pseudocontractive}
 if $\lambda\equiv0$ in (\ref{eq:01}). Obviously, each $\lambda$-strictly pseudocontractive mapping is a Lipschitzian and pseudocontractive mapping with $L=\frac{\lambda+1}{\lambda}$. In particular, a
nonexpansive mapping is $\lambda$-strictly
pseudocontractive mapping in a Hilbert space, but the conversion may be false.

 For finding a fixed point of $\lambda$-strictly
pseudocontractive mapping $T$,  a strong convergence theorem was obtained by Zhou \cite{Z08} in a 2-uniformly smooth Banach space.
\vskip 2mm

\begin{themz} \label{th:z}  (Zhou \cite[Theorem 2.3]{Z08}) Let $C$ be a closed convex subset of a real 2-uniformly smooth Banach space $E$ and let $T : C\to C$ be a
$\lambda$-strict pseudo-contraction with $F(T )\neq\emptyset$. Given
$u,x_0\in C$, a sequence $\{x_n\}$ is generated by
\begin{equation}\label{eq:02}x_{n+1}=\beta_{n}u+\gamma_nx_n+(1-\beta_{n}-\gamma_n)[\alpha_{n}Tx_n+(1-\alpha_{n})x_n],\end{equation}  where $\{\alpha_{n}\}$, $ \{\beta_{n}\}$ and $\{\gamma_n\}$
 in $(0,1)$ satisfy:

(i) $ \alpha_{n}\in [a, \mu]$, $\mu=\min\{1,\frac{\lambda}{K^2}\}$ for some constant $a\in (0,\mu)$;

(ii) $\lim\limits_{n\to\infty}\beta_n= 0$ and
$\sum\limits_{n=1}^\infty\beta_n=\infty$;

(iii) $\lim\limits_{n\to\infty}|\alpha_{n+1}-\alpha_n|= 0;$

(iv)
$0<\liminf\limits_{n\to\infty}\gamma_n\leq\limsup\limits_{n\to\infty}\gamma_n<1$.\\
Then the sequence
$\{x_n\}$ converges strongly to  a fixed point $z$ of $T$.\end{themz}

Recently, Zhang and Su \cite{ZS} extended Zhou's results to
$q$-uniformly smooth Banach space. We also note that the above results
excluded  $\gamma_n\equiv0$ and $\gamma_n=\frac1{n+1}$. Very recently, Chai and Song \cite{CS2011} studied the strong convergence of the modified Mann's iteration (\ref{eq:02}) with $\gamma_n\equiv0$.
\vskip 2mm

 \begin{themcs} \label{tm:cs} (Chai and Song \cite[Theorem 3.1]{CS2011})   Let $C$ be a closed convex subset of a real 2-uniformly smooth Banach space $E$ and let $T : C\to C$ be a
$\lambda$-strict pseudo-contraction with $F(T )\neq\emptyset$. Given
$u,x_0\in C$, a sequence $\{x_n\}$ is generated by
\begin{equation}\label{eq:03}x_{n+1}=\beta_{n}u+(1-\beta_{n})[\alpha_{n}Tx_n+(1-\alpha_{n})x_n],\end{equation}  where $\{\alpha_{n}\}$ and $ \{\beta_{n}\}$ in $(0,1)$ satisfy the following
control conditions:

(i) $ \alpha_{n}\in [a, \mu]$, $\mu=\min\{1,\frac{\lambda}{K^2}\}$ for some constant $a\in (0,\mu)$;

(ii) $\sum\limits_{n=1}^\infty|\alpha_{n+1}-\alpha_n|<\infty$;

(iii) $\lim\limits_{n\to\infty}\beta_n= 0$,
$\sum\limits_{n=1}^\infty\beta_n=\infty$ and
$\sum\limits_{n=1}^\infty|\beta_{n+1}-\beta_n|<\infty$.\\
Then, the sequence $\{x_n\}$ converges strongly to a fixed point $z$ of $T$. \end{themcs}\vskip 2mm

In this paper,  we will deal with strong convergence of the modified Mann's iteration (\ref{eq:02}) under more relaxed conditions on the sequences $\{\alpha_{n}\}$, $ \{\beta_{n}\}$ and $\{\gamma_n\}$
 in $(0,1)$,

(i) $\alpha_{n}\in [0, \mu]$, $\mu=\min\{1,\frac{\lambda}{K^2}\}$ with $\liminf\limits_{n\to\infty}\alpha_n(\lambda-K^2\alpha_n)> 0$;

(ii) $\lim\limits_{n\to\infty}\beta_n= 0$ and
$\sum\limits_{n=1}^\infty\beta_n=\infty$;

(iii) $\limsup\limits_{n\to\infty}\gamma_n<1$. \\
Our results obviously develop and complement  the corresponding
 ones of Zhou \cite{Z08}, Song and Chai \cite{SX09},  Chai and Song \cite{CS2011},   Zhang and Su \cite{ZS} and others. Moreover, our conditions are simpler, which contain $\gamma_n\equiv0$ and $\gamma_n=\frac1{n+1}$ as special cases. Our conclusions may be regarded as  a unification of the some existing results.\vskip 2mm

\section{\bf  Preliminaries and basic results}
For achieving our purposes,  the following facts and results are needed.
 Let $\rho_E : [0,\infty)\to[0,\infty)$ be the
modulus of smoothness of $E$ defined by $$\rho_E(t) =
\sup{\left\{\frac{1}{2}(\|x + y\| + \|x - y\|) - 1 : x \in S(E),
\|y\|\leq t\right\}}.$$  Let $q > 1$. A Banach space $E$ is said to
be {\em $q$-uniformly smooth} if there exists a fixed constant $c> 0$
such that $\rho_E(t)\leq ct^q$
and {\em uniformly smooth} if $\lim_{t\to0}\frac{\rho_E(t)}t=0$.
Clearly, a $q$-uniformly smooth space must be uniformly smooth. Typical
example of uniformly smooth Banach spaces is $L_p$ ($p> 1$). More
precisely, $L_p$ is $\min\{p,2\}$-uniformly smooth for every $p> 1$.  \\

\begin{lem} \label{lm:1}{\em(Zhou \cite[Lemma 1.2]{Z08})}  Let $C$ be a
nonempty subset of  a real $2$-uniformly
smooth Banach space $E$ with the best smooth constant $K$, and let $T : C\to C$ be a $\lambda$-strict
pseudocontraction. For any $\alpha \in (0, 1)$, we define $T_\alpha
= (1-\alpha)x+\alpha Tx.$  Then,
\begin{equation}\label{eq:04}\|T_{\alpha}x-T_{\alpha}y\|^2\leq\|x-y\|^2-2\alpha(\lambda-K^2\alpha)\|Tx-Ty-(x-y)\|^2 \mbox{ for all }x,y\in C.\end{equation}
In particular, as $\alpha\in (0,
\frac{\lambda}{K^2}]$, $T_\alpha: C \to C$ is nonexpansive such that
$F(T_\alpha) = F(T)$.\end{lem}

Lemma \ref{lm:2} was shown and used by several authors. For detail proofs, see Liu \cite{LL} and Xu \cite{15,X5}. Furthermore, a variant of Lemma \ref{lm:1} has already been used by Reich in \cite[Theorem 1]{R79}.

 \begin{lem} \label{lm:2}   Let $\{a_n\}$ be
a sequence of nonnegative real numbers such that
 $$a_{n+1}\leq(1-t_n)a_n+t_nc_n, \quad \forall\ n\geq0.$$
Assume that $\{t_n\}\subset [0,1]$ and $\{c_n\}\subset(0,+\infty)$  satisfy the
 restrictions: $$\sum\limits_{n=0}^\infty t_n=\infty\mbox{ and
 }\limsup\limits_{n\to\infty}c_n\leq0.$$
 Then  as $n\to \infty$, $\{a_n\}$ converges to zero. \end{lem}

Morales and Jung \cite{MJ}, in 2000, proved the following behavior for pseudocontractive mappings. Also see  Song and Chen \cite{SC3, SC4} for more details. The same result of nonexpansive mapping was shown by Reich \cite{R1980} in 1980.\\

\begin{lem} \label{lm:3} {\em(\cite{MJ, SC3, SC4})}  Let $C$ be a
nonempty, closed and convex subset of a uniformly smooth Banach space $E$, and let $T:C\to C$ be a continuous
pseudocontractive mapping with
$F(T)\not=\emptyset$. Suppose that for $t\in(0,1)$ and $u\in C$, $x_t$ defined by \begin{equation}\label{eq:05}x_t = tu+(1-t)Tx_t.\end{equation} Then, as $t \to0$, $x_t$ converges strongly to a fixed point of $T$.\end{lem}

This following results play a key role in proving our main results, which was proved by Song and Chen \cite{SC3}.\\

\begin{lem} \label{lm:4} {\em(Song and Chen \cite[Theorm 2.3]{SC3}) } Let $C$ be a
nonempty, closed and convex subset of a uniformly smooth Banach space $E$, and let $T:C\to C$ be a continuous
pseudocontractive mapping with a fixed point. Assume that there
exists a bounded sequence $\{x_n\}$ such that
$\lim\nolimits_{n\to\infty}\|x_n-Tx_n\|=0$ and
$z=\lim\nolimits_{t\to0}z_t$ exists, where $\{z_t\}$ is defined by
(\ref{eq:05}). Then
$$\limsup\limits_{n\to \infty}\langle
u-z,J(x_n-z)\rangle\leq0.$$\end{lem}

We also need  the following results that showed by Mainge in 2008.

\begin{lem}\label{lm:6}{\em (Mainge \cite[Lemma 3.1]{PM08})} Let $\{\Gamma_n\}$ be a sequence of real numbers that does not decrease at infinity,
in the sense that there exists a subsequence $\{\Gamma_{n_k}\}$ of $\{\Gamma_n\}$ such that
$$\Gamma_{n_k}<\Gamma_{n_k+1} \mbox{ for all } k\geq0. $$
Also consider the sequence of integers  $\{\tau(n)\}_{n\geq n_0}$ defined by
$$\tau(n) = \max\{k\leq n; \Gamma_k < \Gamma_{k+1}\}.$$
Then $\tau(n)$ is a nondecreasing sequence verifying
$$\lim\limits_{n\to\infty}\tau(n)=+\infty ,$$
and, for all $n\geq n_0$, the following two estimates hold:
$$\Gamma_{\tau(n)}\leq \Gamma_{\tau(n)+1}\mbox{ and } \Gamma_n\leq \Gamma_{\tau(n)+1}.$$\end{lem}

\section{\bf Main results}

 In this section, we will present our main results of this paper.

\begin{thm} \label{tm:5} Let $C$ be a closed convex subset of a real 2-uniformly smooth Banach space $E$ and let $T : C\to C$ be a
$\lambda$-strict pseudo-contraction with $F(T )\neq\emptyset$. Given
$u,x_0\in C$, a sequence $\{x_n\}$ is generated by the modified Mann's iteration (\ref{eq:02}), where $\{\alpha_{n}\}$, $ \{\beta_{n}\}$ and $\{\gamma_n\}$
 in $(0,1)$ satisfy:

(i) $\alpha_{n}\in [0, \mu]$, $\mu=\min\{1,\frac{\lambda}{K^2}\}$ with $\liminf\limits_{n\to\infty}\alpha_n(\lambda-K^2\alpha_n)> 0$;

(ii) $\lim\limits_{n\to\infty}\beta_n= 0$ and
$\sum\limits_{n=1}^\infty\beta_n=\infty$;

(iii) $\limsup\limits_{n\to\infty}\gamma_n<1$. \\
Then the sequence
$\{x_n\}$ converges strongly to  a fixed point $z$ of $T$.\end{thm}
\begin{proof}  Let $y_n=T_{\alpha_n}x_n=\alpha_nTx_n+(1-\alpha_n)x_n$.  Then for each $n$, $T_{\alpha_n}$ is nonexpansive and $F(T)=F(T_{\alpha_n})$ by Lemma \ref{lm:1}. So, the sequence $\{x_n\}$ is bounded since for given $p\in F(T)=F(T_{\alpha_n})$,
\begin{align}
 \|x_{n+1}-p\|&=\|\beta_n(u-p)+\gamma_n(x_n-p)+(1-\beta_n-\gamma_n)(T_{\alpha_n}x_n-p)\|
\nonumber\\&\leq\beta_n\|u-p\|+\gamma_n\|x_n-p\|+(1-\beta_n-\gamma_n)\|T_{\alpha_n}x_n-T_{\alpha_n}p\|\nonumber\\
&\leq \beta_n\|u-p\|+\gamma_n\|x_n-p\|+(1-\beta_n-\gamma_n)\|x_n-p\|\nonumber\\
&\leq \beta_n\|u-p\|+(1-\beta_n)\|x_n-p\|\nonumber\\
&\leq \max\{\|x_n-p\|,\|u-p\|\}\nonumber\\
&\ \ \vdots\nonumber\\
&\leq \max\{\|x_0-p\|,\|u-p\|\}.\nonumber\end{align}

Now we show $\lim\limits_{n\to\infty}\|x_n-Tx_n\|=0$.
It follows from Lemma \ref{lm:1} that
\begin{equation}\label{eq:06}\|y_n-p\|=\|T_{\alpha_n} x_n-p\|^2\leq\|x_n-p\|^2-2\alpha_n(\lambda-K^2\alpha_n)\|x_n-Tx_n\|^2.\end{equation}
Furthermore, we also have
 \begin{align}
 \|&x_{n+1}-p\|^2=\|\beta_n(u-p)+\gamma_n(x_n-p)+(1-\beta_n-\gamma_n)(y_n-p)\|^2\nonumber\\
 \leq&\beta_n\|u-p\|^2+\gamma_n\|x_n-p\|^2\nonumber\\& +(1-\beta_n-\gamma_n)(\|x_n-p\|^2-2\alpha_n(\lambda-K^2\alpha_n)\|x_n-Tx_n\|^2)\nonumber\\
 \leq&\beta_n\|u-p\|^2+(1-\beta_n)\|x_n-p\|^2\nonumber\\
 &-2\alpha_n(1-\beta_n-\gamma_n)(\lambda-K^2\alpha_n)\|x_n-Tx_n\|^2\nonumber\\
 \leq&\|x_n-p\|^2-(2\alpha_n(1-\beta_n-\gamma_n)(\lambda-K^2\alpha_n)\|x_n-Tx_n\|^2-\beta_n\|u-p\|^2).\nonumber
 \end{align}
 Then we obtain
$$2\alpha_n(1-\beta_n-\gamma_n)(\lambda-K^2\alpha_n)\|x_n-Tx_n\|^2\leq\|x_n-p\|^2-\|x_{n+1}-p\|^2+\beta_n\|u-p\|^2.$$

It follows from Lemma \ref{lm:3} that there exist $z\in F(T)$ and $x_t = tu+(1-t)Tx_t$ such that $\lim\limits_{t\to0}x_t=z.$  Then we also have \begin{equation}\label{eq:0 7}2\alpha_n(1-\beta_n-\gamma_n)(\lambda-K^2\alpha_n)\|x_n-Tx_n\|^2\leq\|x_n-z\|^2-\|x_{n+1}-z\|^2 +\beta_n\|u-z\|^2.\end{equation} Following the proof technique in Mainge \cite[Lemma 3.2, Theorem 3.1]{PM08},  the proof may be divided two cases.

 {\bf Case 1.} If there exists $N_0$ such that the sequence $\{\|x_n-z\|^2\}$  is nonincreasing for $n\geq N_0$, then the limit $\lim\limits_{n\to\infty}\|x_n-z\|^2$ exists, and hence $\lim\limits_{n\to\infty}(\|x_n-z\|^2-\|x_{n+1}-z\|^2)=0.$ So  by the condition (ii) and the inquality (\ref{eq:0 7}), it is obvious that  $$\limsup\limits_{n\to\infty}\alpha_n(1-\beta_n-\gamma_n)(\lambda-K^2\alpha_n)\|x_n-Tx_n\|^2=0.$$
 It follows from the conditions (i), (ii) and (iii) that  \begin{equation} \label{eq:07}\lim\limits_{n\to\infty}\|x_n-Tx_n\|=0.\end{equation}
 Then by Lemma \ref{lm:4}, we obtain  \begin{equation}\label{eq:08}\limsup\limits_{n\to\infty}\langle u-z
,J(x_{n+1}-z)\rangle\leq0.\end{equation}

Finally, we show that $x_n\to z$. Indeed, since
 \begin{align}\|x_{n+1}-z\|^2=&\langle(\beta_nu+\gamma_nx_n+
 (1-\beta_n-\gamma_n)T_{\alpha_n}x_n)-z,J(x_{n+1}-z)\rangle\nonumber\\
 \leq & \beta_n\langle u-z,J(x_{n+1}-z)\rangle+\gamma_n\|x_n-z\|\|J(x_{n+1}-z)\|\nonumber\\&+
 (1-\beta_n-\gamma_n)\|T_{\alpha_n}x_n-z\|\|J(x_{n+1}-z)\|\nonumber\\
 \leq & \beta_n\langle u-z,J(x_{n+1}-z)\rangle+(1-\beta_n)\|x_n-z\|\|x_{n+1}-z\|\nonumber\\
 \leq & \beta_n\langle u-z,J(x_{n+1}-z)\rangle+
 (1-\beta_n)\frac{\|x_n-z\|^2+\|x_{n+1}-z\|^2}2\nonumber,\nonumber
\end{align}
then, we have
\begin{equation}\label{eq:09}\|x_{n+1}-z\|^2\leq
(1-\beta_n)\|x_n-z\|^2+2\beta_n\langle u-z,J(x_{n+1}-z)\rangle.\end{equation}
So, an application of Lemma \ref{lm:2} onto (\ref{eq:09})  yields that $\lim\limits_{n\to\infty}\|x_n- z\|=0$.

{\bf Case 2.} Assume that there exists a subsequence $\{\|x_{n_k}-z\|^2\}$ of $\{\|x_n-z\|^2\}$  such that $\|x_{n_k}-z\|^2<\|x_{n_k+1}-z\|^2$ for for all $k\geq0$. Let $$ \Gamma_n=\|x_n-z\|^2\mbox{ and } \tau(n) = \max\{k\leq n; \Gamma_k < \Gamma_{k+1}\}.$$
 
 It follows from Lemma \ref{lm:6} that $\tau(n)$ is a nondecreasing sequence verifying
$$\lim\limits_{n\to\infty}\tau(n)=+\infty $$
and for $n$ large enough,
\begin{equation}\label{eq:10}\Gamma_{\tau(n)}\leq \Gamma_{\tau(n)+1},\ \Gamma_n=\|x_n-z\|^2\leq \Gamma_{\tau(n)+1}.\end{equation}
In light of Eq. (\ref{eq:0 7}), we have $$2\alpha_{\tau(n)}(1-\beta_{\tau(n)}-\gamma_{\tau(n)})(\lambda-K^2\alpha_{\tau(n)})\|x_{\tau(n)}-Tx_{\tau(n)}\|^2\leq\beta_{\tau(n)}\|u-z\|^2,$$ and so by  the condition (i),(ii) and (iii), we have $$\lim_{n\to\infty}\|x_{\tau(n)}-Tx_{\tau(n)}\|=0.$$
 Then as $n\to\infty$, $$\begin{aligned}\|x_{\tau(n)+1}-Tx_{\tau(n)}\|\leq&\beta_{\tau(n)}\|u-Tx_{\tau(n)}\|+\gamma_{\tau(n)} \|x_{\tau(n)}-Tx_{\tau(n)}\|\\
 & +(1-\beta_{\tau(n)}-\gamma_{\tau(n)})(1-\alpha_{\tau(n)})\|x_{\tau(n)}-Tx_{\tau(n)}\|\to 0.\end{aligned}$$ Since $$\begin{aligned}\|x_{\tau(n)+1}-Tx_{\tau(n)+1}\|&\leq\|x_{\tau(n)+1}-Tx_{\tau(n)}\|+\|Tx_{\tau(n)}-Tx_{\tau(n)+1}\|\\
&\leq\|x_{\tau(n)+1}-Tx_{\tau(n)}\|+\|x_{\tau(n)}-x_{\tau(n)+1}\|\\
&\leq2\|x_{\tau(n)+1}-Tx_{\tau(n)}\|+\|x_{\tau(n)}-Tx_{\tau(n)}\|,
\end{aligned}$$ we have \begin{equation} \label{eq:11}\lim_{n\to\infty}\|x_{\tau(n)+1}-Tx_{\tau(n)+1}\|=0.\end{equation}
Then by Lemma \ref{lm:4}, we obtain  \begin{equation}\label{eq:13}\limsup\limits_{n\to\infty}\langle u-z
,J(x_{\tau(n)+1}-z)\rangle\leq0.\end{equation}
Using the similar proof techniques of Case 1, the only modification is that $n$ is replaced by $\tau(n)$, we have \begin{equation}\label{eq:14}\|x_{\tau(n)+1}-z\|^2\leq
(1-\beta_{\tau(n)})\|x_{\tau(n)}-z\|^2+2\beta_{\tau(n)}\langle u-z,J(x_{\tau(n)+1}-z)\rangle.\end{equation}
Together with (\ref{eq:10}), we have $$\Gamma_{\tau(n)}\leq \Gamma_{\tau(n)+1}\leq
(1-\beta_{\tau(n)})\Gamma_{\tau(n)}+2\beta_{\tau(n)}\langle u-z,J(x_{\tau(n)+1}-z)\rangle,$$ and so,
$$\Gamma_{\tau(n)}=\|x_{\tau(n)}-z\|^2\leq2\langle u-z,J(x_{\tau(n)+1}-z)\rangle.$$
Along with (\ref{eq:13}), we have $$\lim\limits_{n\to\infty}\Gamma_{\tau(n)}=\lim\limits_{n\to\infty}\|x_{\tau(n)}- z\|=0.$$
It follows from (\ref{eq:14}), (\ref{eq:13}) and the condition (ii) that
$$\lim\limits_{n\to\infty}\Gamma_{\tau(n)+1}=\lim_{n\to\infty}\|x_{\tau(n)+1}-z\|=0.$$  Now it follows from (\ref{eq:10}) that $$\lim\limits_{n\to\infty}\Gamma_n=\lim\limits_{n\to\infty}\|x_n- z\|=0.$$ The proof is completed.
\end{proof}

Clearly, Theorem \ref{tm:5} contains $\gamma_n\equiv0$ and $\gamma_n=\frac1{n+1}$ as special cases. So the following result is obtained easily.

\begin{cor} \label{co:6} Let $C$ be a closed convex subset of a real 2-uniformly smooth Banach space $E$ and let $T : C\to C$ be a $\lambda$-strict pseudo-contraction with $F(T )\neq\emptyset$. Given
$u,x_0\in C$, a sequence $\{x_n\}$ is generated by the modified Mann's iteration (\ref{eq:03}), where $\{\alpha_{n}\}$ and $ \{\beta_{n}\}$ in $(0,1)$ satisfy:

(i) $\alpha_{n}\in [0, \mu]$, $\mu=\min\{1,\frac{\lambda}{K^2}\}$ with $\liminf\limits_{n\to\infty}\alpha_n(\lambda-K^2\alpha_n)> 0$;

(ii) $\lim\limits_{n\to\infty}\beta_n= 0$ and
$\sum\limits_{n=1}^\infty\beta_n=\infty$. \\
Then the sequence
$\{x_n\}$ converges strongly to  a fixed point $z$ of $T$.\end{cor}

Using the same proof techniques as Theorem \ref{tm:5}, we easily obtain the following result. Since the only difference is that $\alpha_n(\lambda-K^2\alpha_n)$ is replaced by $\alpha_n(q\lambda-C_q\alpha_n^{q-1})$ in its proof (Lemma 2.1 is replaced by Lemma 2.2 of Zhang and  Su \cite{ZS}), so we omit its proof. \\

\begin{thm} \label{tm:7} Let $C$ be a closed convex subset of a real $q-$uniformly smooth Banach space $E$ ($q>1$) and let $T : C\to C$ be a
$\lambda$-strict pseudo-contraction with $F(T )\neq\emptyset$. Given
$u,x_0\in C$, a sequence $\{x_n\}$ is generated by the modified Mann's iteration (\ref{eq:02}) or (\ref{eq:03}), where $\{\alpha_{n}\}$, $ \{\beta_{n}\}$ and $\{\gamma_n\}$
 in $(0,1)$ satisfy:

(i) $\alpha_{n}\in [0, \mu]$, $\mu=\min\{1,\{\frac{q\lambda}{C_q}\}^{\frac1{q-1}}\}$ with $\liminf\limits_{n\to\infty}\alpha_n(q\lambda-C_q\alpha_n^{q-1})> 0$;

(ii) $\lim\limits_{n\to\infty}\beta_n= 0$ and
$\sum\limits_{n=1}^\infty\beta_n=\infty$;

(iii) $\limsup\limits_{n\to\infty}\gamma_n<1$. \\
Then the sequence
$\{x_n\}$ converges strongly to  a fixed point $z$ of $T$.\end{thm}

{\bf Remark 1.} If $E$ is $q-$uniformly smooth, then $1<q\leq 2$ and
$E$ is uniformly smooth\cite{KTY}, and hence Theorem \ref{tm:5}  may be regarded  as a special case of Theorem \ref{tm:7}.

{\bf Remark 2.} In Theorem \ref{tm:5} and Corollary \ref{tm:7}, the sequence $\{\gamma_n\}$ only need satisfy $\limsup\limits_{n\to\infty}\gamma_n<1$. Then Theorem \ref{tm:5} may properly contain Theorem 2.3 of Zhou \cite{Z08} and Theorem 3.1 of Chai and Song \cite{CS2011} as a special case,  and Theorem 4.1 of Zhang and  Su \cite{ZS} may be obtained from Corollary \ref{tm:7}. So our conclusions may be regarded as a unification of the some existing results.

{\bf Remark 3.} Our main results are obtained in the frame of $q-$uniformly smooth Banach space, then in the future work, we may consider the results of this paper in $n-$Banach space. For more details on $n-$Banach space, see Dutta \cite{HD1,HD2,HD3,HD4,HD5,HD6,HD7,HD8,HD9}.

\section*{\bf Acknowledgments}
The authors are grateful to the anonymous referee for his/her valuable
suggestions which helps to improve this manuscript.

\end{document}